\documentclass[letterpaper, 10 pt, conference]{ieeeconf}  

\IEEEoverridecommandlockouts                              
\usepackage{graphics} 
\usepackage{amsmath,amssymb,amsfonts}

\usepackage[justification=centering]{caption}
\usepackage{subfigure}
\usepackage{verbatim}
\usepackage{graphicx}
\newtheorem{assumption}{ {Assumption}}
\newtheorem{theorem}{ {Theorem}}
\newtheorem{lemma}{ {Lemma}}
\newtheorem{definition}{ {Definition}}
\newtheorem{remark}{ {Remark}}

\newcommand{\sign}[1]{\mathrm{sgn}(#1)}

\title{\LARGE \bf
Continuous-Time and Event-Triggered Online Optimization for Linear Multi-Agent Systems
}

\author{Yang Yu, Xiuxian Li, \IEEEmembership{Senior Member, IEEE}, Li Li, and Lihua Xie, \IEEEmembership{Fellow, IEEE}
\thanks{Yang Yu, Xiuxian Li, and Li Li are with the Department of Control Science and Engineering, Shanghai Research Institute for Autonomous Intelligent Systems, and Shanghai Institute of Intelligent Science and Technology, Tongji University, Shanghai, China {\tt\small \{1910639, xli, lili\}@tongji.edu.cn}}
\thanks{Lihua Xie is with the School of Electrical and Electronic Engineering, Nanyang Technological University, Singapore 639798 {\tt\small elhxie@ntu.edu.sg}}}

\begin{document}

\maketitle
\thispagestyle{empty}
\pagestyle{empty}

\begin{abstract}

This paper studies the decentralized online convex optimization problem for heterogeneous linear multi-agent systems. Agents have access to their time-varying local cost functions related to their own outputs, and there are also time-varying coupling inequality constraints among them. The goal of each agent is to minimize the global cost function by selecting appropriate local actions only through communication between neighbors. We design a distributed controller based on the saddle-point method which achieves constant regret bound and sublinear fit bound. In addition, to reduce the communication overhead, we propose an event-triggered communication scheme and show that the constant regret bound and sublinear fit bound are still achieved in the case of discrete communications with no Zeno behavior. A numerical example is provided to verify the proposed algorithms.


\end{abstract}

\section{INTRODUCTION}

Convex optimization has been widely studied as a pretty effective method in research fields involving optimization and decision-making, such as automatic control systems \cite{ran1996}, communication networks \cite{luo2003}, and machine learning \cite{qiu2016}. 
Early convex optimization works were based on fixed cost functions and static constraints.
However, in practice, optimization costs and constraints of many problems are possible to be time-varying and a priori unknown \cite{ss2011}. This motivated  online convex optimization (OCO) which requires the decision maker to choose an action at each instant based on previous information. 
A widely used performance criterion of OCO is regret, that is, the gap between the cumulative loss of the selected action and that of the best ideal action made when knowing the global information beforehand. If regret is sublinear, the time average loss of the selected action is progressively not greater than that of the ideal action. Another performance indicator is fit, which measures the degree of violation of static/time-varying inequality constraints. For more details, a recent survey can be referenced \cite{li2022sur}.

The OCO framework was introduced by \cite{zink2003}, where the projection-based online gradient descent algorithm was analyzed. Based on static constraints, the algorithm was proved to achieve $\mathcal{O}(\sqrt{T})$ static regret bound for time-varying convex cost functions with bounded subgradients. 
With the increase of data scale and problem complexity in recent years, distributed online convex optimization has also been widely studied in recent years \cite{zhou2018}. In the continuous-time setting, the saddle-point algorithm proposed in \cite{lee2016} under constant constraints is shown to achieve sublinear bounds on the network disagreement and the regret achieved by the strategy. The authors of \cite{pater2020} generalized this result to the problem of time-varying constraints.
In the discrete-time setting, \cite{yuan2018, lxx2021} used distributed primal-dual algorithms to solve online convex optimization with static independent and coupled inequality constraints. In order to solve the time-varying coupling constraints, the authors of \cite{yi2020} proposed a novel distributed online primal dual dynamic mirror descent algorithm to realize sublinear dynamic regret and constraint violation. 
A gradient-free distributed bandit online algorithm was proposed in \cite{yi20212}, which is applicable to scenarios where it is difficult to obtain the gradient information of the cost functions.
In the presence of aggregate variables in local cost functions, an online distributed gradient tracking algorithm was developed in \cite{li20212} based on true or stochastic gradients.


In actual physical systems, the implementation of optimization strategies must take into account the complicated dynamics of each agent. Along this line, only a few works have investigated online convex optimization with physical systems in recent years. 
For continuous-time multi-agent systems with high-order integrators, the authors of \cite{deng2016} used PI control idea and gradient descent to solve the distributed OCO problem. The authors of \cite{non2020} considered the online convex optimization problem of linear systems, but did not consider any constraints. The online convex optimization problem of linear time-invariant (LTI) system was studied in \cite{non2021} based on the behavioral system theory, where a proposed data-driven algorithm that does not rely on the model achieves sublinear convergence. However, the above two papers for linear systems only provide centralized algorithms. The distributed setup for online optimization algorithm with linear systems is yet to be studied. 

The main contributions of this paper are as follows.
\begin{itemize}
	\item Compared with the centralized OCO algorithms for linear systems with no constrains \cite{non2020,non2021}, this paper studies the distributed online optimization of heterogeneous multi-agent systems with time-varying coupled inequality constraints for the first time. Agents only rely on the information of themselves and their neighbors to make decisions and achieve constant regret bound and $\mathcal{O}(\sqrt{T})$ fit bound. In comparison, most existing algorithms \cite{yi2020,li2022} about distributed online optimization with coupled inequality constraints only achieve inferior sublinear regret bounds.
	\item Compared with current continuous-time online optimization algorithms \cite{deng2016, lee2016, pater2020}, this paper introduces an event-triggered mechanism to reduce the communication overhead. In the case of discrete communication, the constant regret bound and $\mathcal{O}(\sqrt{T})$ fit bound are still achieved.
\end{itemize}

The rest of the paper is organized as follows. Preliminaries are given in Section \ref{pre}. In Section \ref{pro}, the heterogeneous multi-agent system under investigation is described mathematically, the online convex optimization problem is defined and some useful lemmas are given. Following that, the control laws with continuous and event-triggered communication are proposed, respectively, and the constant regret bound and sublinear fit bound are established in Section \ref{mai}. Then, a simulation example is provided to verify the effectiveness of the algorithm in Section \ref{sim}. Finally, the conclusion is discussed in Section \ref{res}.

\section{PRELIMINARIES}\label{pre}

\subsection{Notations}
Let $\mathbb{R}$, $\mathbb{R}^n$, $\mathbb{R}_+^n$, $\mathbb{R}^{m\times n}$ be the sets of  real numbers, real vectors of dimension $n$, non-negative real vectors of dimension $n$, and real matrices of dimension $m\times n$, respectively.
The $n\times n$ identity matrix is denoted by $I_n$. The $n\times 1$ all-one and all-zero column vectors are denoted by $\boldsymbol1_n$ and $\boldsymbol0_n$, respectively.
For a matrix $A\in\mathbb{R}^{m \times n}$, $A^\top $ is its transpose and  $diag(A_1,\dots,A_n)$ denotes a block diagonal matrix with diagonal blocks of $A_1$, $\dots$, $A_n$.
For a vector $x$, $\|x\|_1$ is its $1$-norm, $\|x\|$ is its $2$-norm, and $col(x_1,\dots,x_n)$ is a column vector by stacking vectors $x_1,\dots,x_n$. $A\otimes B$ represents the Kronecker product of matrices $A$ and $B$. 
Let $P_\mathcal{S}(x)$ be the Euclidean projection of a vector $x\in \mathbb{R}^n$ onto the set $\mathcal{S} \subseteq \mathbb{R}^n$, i.e., 
$P_\mathcal{S}(x)=argmin_{y\in\mathcal{S}}\|x-y\|^2$. For simplicity, let $[\cdot]_+$ denote $P_{\mathbb{R}_+^n}(\cdot)$. Define the set-valued sign function $\mathrm{sgn}(\cdot)$ as follows:
\begin{align*}
	\sign{x}:=\partial\|x\|_1=
	\begin{cases}
		-1,&\mathrm{if}\  x<0,\\
		[-1,1],&\mathrm{if}\  x=0,\\
		1,&\mathrm{if}\  x>0.
	\end{cases}
\end{align*}

\subsection{Graph Theory}
For a system with $N$ agents, its communication network is modeled by an undirected graph $\mathcal{G} = (\mathcal{V}, \mathcal{E})$, where $\mathcal{V} = \{v_1, \dots, v_N\}$ is a node set and $\mathcal{E} \in \mathcal{V} \times \mathcal{V}$ is an edge set. If information exchange can occur between $v_j$ and $v_i$, then $(v_j,v_i) \in \mathcal{E}$ with $a_{ij}=1$ denoting its weight.  $\mathcal{A}= [a_{ij}]\in \mathbb{R}^{N \times N}$ is the adjacency matrix. If there exists a path from any node to any other node in $\mathcal{V}$, then $\mathcal{G}$ is called connected.


\section{PROBLEM FORMULATION}\label{pro}

Consider a multi-agent system consisting of $N$ heterogeneous agents indexed by $i=1, \dots, N$, and the $i$th agent has following linear dynamics:
\begin{align} 
	\begin{split}
		{\dot x}_i&=A_i x_i + B_i u_i,\\
		y_i&=C_i x_i,\label{eqsys}
	\end{split}
\end{align}
where $x_i\in\mathbb{R}^{n_i}$, $u_i\in\mathbb{R}^{m_i}$ and $y_i\in\mathbb{R}^{p_i}$ are the state, input and output variables, respectively. $A_i\in\mathbb{R}^{n_i \times n_i}$, $B_i\in\mathbb{R}^{n_i \times m_i}$ and $C_i\in\mathbb{R}^{p_i \times n_i}$ are the state, input and output matrices, respectively.

Each agent $i$ has an output set $\mathcal{Y}_i\subseteq\mathbb{R}^{p_i}$ such that the output variable $y_i\in\mathcal{Y}_i$. $f_i(t,\cdot): \mathbb{R}^{p_i}\to \mathbb{R}$ and $g_i(t,\cdot): \mathbb{R}^{p_i}\to \mathbb{R}^{q}$ are the private cost and constraint functions for agent $i$.
Denote $p:=\sum_{i=1}^{N}p_i$, $\mathcal{Y}:=\mathcal{Y}_1\times\dots \times \mathcal{Y}_N \subseteq \mathbb{R}^{p}$, $y:=col(y_1,\dots, y_N)\in\mathbb{R}^p$, and $f(t,y):=\sum_{i=1}^{N}f_{i}(t, y_i)$.
The objective of this paper is to design a controller $u_i(t)$ for each agent by using only local interaction and information such that all agents cooperatively minimize the sum of the cost functions over a period of time $[0,T]$ with time-varying coupled inequality constraints: 
\begin{align}
	\begin{split}
		\min_{y \in \mathcal{Y}} \int_{0}^{T}f(t,y)\,dt ,\\
		s.t. \sum_{i=1}^{N}g_{i}(t, y_i)\le \boldsymbol0.\label{question2}
	\end{split}
\end{align}

Let $y^*=(y_1^*,\dots,y_N^*)\in \mathcal{Y}$ denote the optimal solution for problem (\ref{question2}) when the time-varying cost and constraint functions are known in advance.  

In order to evaluate the cost performance of such output trajectories, we define two performance indicators: network regret and network fit. According to the previous definition, $y^*$ is the optimal output when the agents know all the information of network in the period of $[0, T]$. But in reality, agents can only make decisions based on their own and neighbors' current and previous information. Regret is described as the gap between the cumulative action $\int_{0}^{T} f(t,y(t))\,dt$ incurred by $y(t)$ and the cost incurred by the optimal output $y^*$, i.e.,
\begin{align}
	\mathcal{R}^T:=\int_{0}^{T} \Big(f(t,y(t))-f(t,y^*)\Big)\,dt\label{regert}.
\end{align}

In order to evaluate the fitness of output trajectories $y(t)$ to the constraints (or in other words, the degree of violation of the constraints), we define fit as the projection of the cumulative constraints onto the nonnegative orthant:
\begin{align}
	\mathcal{F}^T:=\left\|\,\left[\int_{0}^{T} \sum_{i=1}^{N}g_{i}(t, y_i) \,dt\right]_+\right\|.\label{fit}
\end{align}
This definition implicitly allows strictly feasible decisions to compensate for violations of constraints at certain times. This is reasonable when the variables can be stored or preserved, such as the average power constraints \cite{pater2017}.
By $g_i(t,\cdot): \mathbb{R}^{p_i}\to \mathbb{R}^{q}$, one can define $g_{i,j}(t,\cdot): \mathbb{R}^{p_i}\to \mathbb{R}$ as the $j$th component of $g_i(t,\cdot)$, i.e., $g_i(t,\cdot) = col\left(g_{i,1}(t,\cdot), \dots, g_{i,q}(t,\cdot)\right)$. Further, define $F_j^T:=\int_{0}^{T} \sum_{i=1}^{N}g_{i,j}(t, y_i)\,dt, j=1,\dots,q$ as the $j$th component of the constraint integral. It can be easily deduced that
\begin{align}
	\mathcal{F}^T=\sqrt{\sum_{j=1}^q\left[F_j^T\right]_+^2}.\notag
\end{align}
\begin{assumption} \label{asp1}
	The communication network $\mathcal{G}$ is undirected and connected.	
\end{assumption}

\begin{assumption} \label{asp2}
	Each set $\mathcal{Y}_i$ is convex and compact. For $t\in [0,T]$, functions $f_i(t,y_i)$ and $g_i(t,y_i)$ are convex, integrable and bounded on $\mathcal{Y}_i$, i.e., there exist constants $K_f>0$ and $K_g>0$ such that $|f_i(t,y_i)|\le K_f$ and $\|g_i(t,y_i)\|\le K_g$.
\end{assumption}

\begin{assumption}
	The set of feasible outputs $\mathcal{Y}^\dagger:=\{y: y\in\mathcal{Y}, \sum_{i=1}^{N}g_{i}(t, y_i)\le 0, t\in[0,T]\}$ is non-empty.
\end{assumption}	

\begin{definition} [\cite{zhang1995}]
	Let $\mathcal{S} \subseteq \mathbb{R}^n$ be a closed convex set. Then, for any $x \in \mathcal{S}$ and $v\in\mathbb{R}^n$, the projection of $v$ over set $\mathcal{S}$ at the point $x$ can be defined as
	\begin{align}
		\Pi_\mathcal{S} [x,v] = \lim\limits_{\xi\to0^+} \frac{P_\mathcal{S}(x+\xi v)-x}{\xi}.\notag
	\end{align}
\end{definition}

\begin{lemma} [\cite{pater2017}]
	Let $\mathcal{S} \subseteq \mathbb{R}^n$ be a convex set and let $x, y\in \mathcal{S}$, then
	\begin{align}
		(x-y)^\top \Pi_{\mathcal{S}}(x,v)\le (x-y)^\top v, \forall v\in \mathcal{S}.\label{eqle2}
	\end{align}
\end{lemma}


\begin{assumption} \label{asp4}
	$(A_i,B_i)$ is controllable, and
	\begin{align}
		rank(C_iB_i)=p_i, i=1,\dots,N.\notag
	\end{align}
\end{assumption}

\begin{lemma} [\cite{yu2021}]\label{lemma3}
	Under Assumption \ref{asp4}, the matrix equations
	\begin{subequations}\label{eqle3}
		\begin{align}
			C_iB_iK_{\alpha_i}&=C_iA_i,\label{eqle31}\\
			C_iB_iK_{\beta_i}&=I_{p_i},\label{eqle32}
		\end{align}
	\end{subequations}
	have solutions $K_{\alpha_i}$, $K_{\beta_i}$.
\end{lemma}

\begin{remark}
	Assumption \ref{asp2} is reasonable since the output variables in practice, such as voltage, often have a certain range. The cost functions and constraint functions are not required to be differentiable, which can be dealt with by using subgradients. The controllability in Assumption \ref{asp4} is quite standard in dealing with the problem for linear systems.
\end{remark}


\section{MAIN RESULTS}\label{mai}

\subsection{Continuous Communication}

For agent $i$, to solve the online optimization problem (\ref{question2}), we can construct the time-varying Lagrangian

\begin{align}
	\mathcal{H}_i(t,y_i,\mu_i)=f_i(t,y_i)+\mu_i^\top g_i(t,y_i) - K_\mu h_i,\label{la}
\end{align}
where $\mu_i\in\mathbb{R}_+^{q}$ is the local Lagrange multiplier for agent $i$, $K_\mu>0$ is the preset parameter and $h_i:= \sum_{j=1}^N a_{ij}\left\|\mu_i-\mu_j\right\|_1$ is a metric of $\mu_i$'s disagreement \cite{liang2018}.

Notice that $f_i(t,\cdot)$, $g_i(t,\cdot)$ are convex and $\mu_i \ge \boldsymbol 0_q$,
hence the Lagrangian is convex with respect to $y_i$. Let us
denote by $\mathcal{H}_i^{y_i}(t,y_i,\mu_i)$ a subgradient of $\mathcal{H}_i$ with respect to $y_i$, i.e.,
\begin{align}
	\mathcal{H}_i^{y_i}(t,y_i,\mu_i)\in \partial f_i(t,y_i)+\mu_i^\top \partial g_i(t,y_i).\label{lax}
\end{align}

The Lagrangian is concave with respect to $\mu_i$ and its subgradient is given by
\begin{align}
	\mathcal{H}_i^{\mu_i}(t,y_i,\mu_i)&\in g_i(t,y_i) - K_\mu\sum_{j=1}^N a_{ij}\sign{\mu_i-\mu_j}.\label{lamu}
\end{align}

For simplicity, define
\begin{align}
	\mathcal{H}(t,y,\mu):= f(t,y)+\mu^\top g(t,y) - K_\mu h(\mu),\label{lag}
\end{align}
where $y=col(y_1,\dots, y_N)$, $\mu=col(\mu_1, \dots, \mu_N)$, $f(t,y)=\sum_{i=1}^{N}f_{i}(t, y_i)$, $g(t,y)=col(g_1(t,y_1), \dots, g_N(t,y_N))$, and $h(\mu)= \sum_{i=1}^N\sum_{j=1}^Na_{ij}\left\|\mu_i-\mu_j\right\|_1$. It can be easily verified that  $\mathcal{H}(t,y,\mu)=\sum_{i=1}^{N}\mathcal{H}_i(t,y_i,\mu_i)$.

A controller following modified Arrow-Hurwicz algorithm for the $i$th agent is proposed as
\begin{subequations}\label{eqpi1}
	\begin{align}
		u_i&=-K_{\alpha_i} x_i+K_{\beta_i} \left(\Pi_{\mathcal{Y}_i}[y_i, -\varepsilon\mathcal{H}_i^{y_i}(t,y_i,\mu_i)]\right),\label{eqpi1a}\\
		\dot{\mu_i}&=\Pi_{\mathbb{R}_+^q}[\mu_i,\varepsilon\mathcal{H}_i^{\mu_i}(t,y_i,\mu_i)], \label{eqpi1b}
	\end{align}
\end{subequations}
where $\varepsilon>0$ is the step size, $K_{\alpha_i},K_{\beta_i}$ are feedback matrices that are the solutions of (\ref{eqle3}), and the initial value $\mu_i(0){=}\boldsymbol{0}$.

Substituting the controller (\ref{eqpi1}) into the system (\ref{eqsys}), the system dynamics of the $i$th agent is
\begin{subequations}\label{eqpi2}
	\begin{align}
		\dot x_i&{=}(A_i{-}B_i K_{\alpha_i})x_i{+}B_iK_{\beta_i}\left(\Pi_{\mathcal{Y}_i}[y_i, -\varepsilon\mathcal{H}_i^{y_i}(t,y_i,\mu_i)]\right),\label{eqpi2a}\\
		\dot{\mu_i}&{=}\Pi_{\mathbb{R}_+^q}[\mu_i,\varepsilon\mathcal{H}_i^{\mu_i}(t,y_i,\mu_i)] ,\label{eqpi2b}\\
		y_i&{=}C_ix_i.\label{eqpi2c}
	\end{align}
\end{subequations}

For the subsequent analysis, consider the following energy function with any $\tilde{y}\in\mathcal{Y}$ and $\tilde{\mu}\in\mathbb{R}_+^{Nq}$ :
\begin{align}
	V_{(\tilde{y}, \tilde{\mu})}(y,\mu)= \frac{1}{2}\|y-\tilde{y}\|^2 + \frac{1}{2}\|\mu-\tilde{\mu}\|^2.\label{v1}
\end{align}
The following lemma establishes the relationship between above energy function and time-varying Larangian (\ref{la}) along the dynamics (\ref{eqpi2}).

\begin{lemma}\label{lemma4}
	If Assumptions \ref{asp1}-\ref{asp4} hold and $\tilde{\mu}:= \boldsymbol1_N \otimes \gamma, \forall \gamma\in \mathbb{R}_+^q$, then for any $T\ge0$ the trajectories of linear multi-agent system (\ref{eqsys}) with control protocol (\ref{eqpi1}) satisfy
	\begin{align}
		\int_{0}^{T} \Big( \mathcal{H}(t,y,\tilde{\mu})-\mathcal{H}(t,\tilde{y},\mu) \Big) \,dt\le\frac{V_{(\tilde{y}, \tilde{\mu})}(y(0), \boldsymbol{0})}{\varepsilon}.\label{l3}
	\end{align}
\end{lemma}

\begin{proof}
	Calculating the time derivative of the energy function (\ref{v1}) together with (\ref{eqpi2}) yields
	
	\begin{align}
		\dot{V}_{(\tilde{y}, \tilde{\mu})}
		{=} 
		&(\mu-\tilde{\mu})^\top \dot{\mu} + (y-\tilde{y})^\top \dot{y}\notag\\
		=
		&\sum_{i=1}^{N} (\mu_i\!-\!\tilde{\mu}_i)\!^{\top} \Pi_{\mathbb{R}_+^q}\big[\mu_i,\!\varepsilon\mathcal{H}_i^{\mu_i}(t,y_i,\mu_i)\big] \notag\\
		& \!+\! \sum_{i=1}^{N} \!(y_i{-}\tilde{y}_i)\!^\top \Pi_{\mathcal{Y}_i}\!\big[y_i,\! \!-\varepsilon\mathcal{H}_i^{y_i}\!(t,y_i,\mu_i)\big] \notag\\
		\le
		& \sum_{i=1}^{N} (\mu_i-\tilde{\mu}_i)^\top \Big(\varepsilon\mathcal{H}_i^{\mu_i}(t,y_i,\mu_i)\!\Big) \notag\\
		& + \sum_{i=1}^{N} (y_i-\tilde{y}_i)^\top \Big(-\varepsilon\mathcal{H}_i^{y_i} (t,y_i,\mu_i)\Big) \notag\\
		\le 
		& \varepsilon \sum_{i=1}^{N} \Big(\mathcal{H}_i(t,y_i,\mu_i) - \mathcal{H}_i(t,y_i,\tilde{\mu}_i)\Big) \notag\\
		& + \varepsilon \sum_{i=1}^{N} \Big(\mathcal{H}_i (t,\tilde{y}_i,\mu_i) - \mathcal{H}_i (t,y_i,\mu_i)\Big) \notag\\
		= 
		&\varepsilon \Big(\mathcal{H}(t,\tilde{y},\mu) - \mathcal{H}(t,y,\tilde{\mu})\Big),\label{v11}
	\end{align}
	where the second equation holds in view of (\ref{eqle3}) and (\ref{eqpi2}), the first inequality  holds because of  (\ref{eqle2}), and the last inequality holds since the Lagrangian (\ref{la}) is convex with respect to $y_i$ and concave with respect to $\mu_i$.
	
	Integrating (\ref{v11}) from $0$ to $T$ on both sides leads to that
	\begin{align}
		&\int_{0}^{T} \Big( \mathcal{H}(t,y,\tilde{\mu})-\mathcal{H}(t,\tilde{y},\mu)\Big) \,dt\notag\\
		\le&-\frac{1}{\varepsilon}\int_{0}^{T} \dot{V}_{(\tilde{y}, \tilde{\mu})} (y(t), \mu(t)) \,dt\notag\\
		=&-\frac{1}{\varepsilon}\Big(V_{(\tilde{y}, \tilde{\mu})}(y(T), \mu(T)) - V_{(\tilde{y}, \tilde{\mu})}(y(0), \mu(0))\Big).
	\end{align}
	Because the energy function (\ref{v1}) is always nonnegative and $\mu(0)=\boldsymbol{0}$, the conclusion (\ref{l3}) can be obtained.
\end{proof}

We now state the main results about the regret and fit bounds of continuous communication controller (\ref{eqpi1}).

\begin{theorem}\label{th1}
	Suppose that Assumptions \ref{asp1}-\ref{asp4} hold. Then for any $T\ge0$ and $\varepsilon>0$ in control protocol (\ref{eqpi1}), by choosing $K_\mu \ge NK_g$, it holds for the regret that
	\begin{align}
		\mathcal{R}^T\le \frac{\|y(0)\!-\!y^*\|^2}{2\varepsilon}.
	\end{align}
\end{theorem}

\begin{proof}
	By choosing $\tilde{y}=y^*$ and $\tilde{\mu}=\boldsymbol{0}_{Nq}$ in Lemma \ref{lemma4}, one has
	\begin{align}
		\int_{0}^{T} \Big(\mathcal{H}(t,y,\boldsymbol{0}_{Nq}) - \mathcal{H}(t,y^*,\mu)\Big)dt\le\frac{V_{(y^*, \boldsymbol{0}_{Nq})}(y(0), \boldsymbol{0})}{\varepsilon}.\label{th1a}
	\end{align}
	According to (\ref{regert}) and (\ref{lag}), it can be obtained that
	\begin{align}
		\mathcal{R}^T{=}
		&\int_{0}^{T} \!\!\left(\mathcal{H}(t,y,\boldsymbol{0}_{Nq}) {-} \mathcal{H}(t,y^*,\mu)\right)dt
		{+}\int_{0}^{T} \!\mu^\top g(t,y^*)\, dt\notag\\
		&-\int_{0}^{T} K_\mu h(\mu)\, dt.\label{th1b}
	\end{align}	
	
	Consider the second term of the right-hand side of (\ref{th1b}). Let $\phi(\mu):=\mu^\top g(t,y^*)$ for simplicity. Then, by introducing an intermediate variable $\bar{\mu}:=\frac{1}{N}\sum_{i=1}^{N}\mu_i$ and the relationship $\sum_{i=1}^{N} g_{i}(t, y_i^*) \le \boldsymbol0$, one has that
	\begin{align}
		\phi(\mu)\le(\mu-\boldsymbol{1}_N\otimes \bar{\mu})^\top g(t,y^*).
	\end{align}
	Further, one can obtain that
	\begin{align}
		\phi(\mu)^2&\le\big((\mu-1_N\otimes \bar{\mu})^\top g(t,y^*)\big)^2\notag\\
		&\le NK_g^2\left\|\mu-1_N\otimes \bar{\mu}\right\|^2\notag\\
		&= NK_g^2 \sum_{i=1}^{N} \Big\|\mu_i-\frac{1}{N}\sum_{j=1}^{N}\mu_j\Big\|^2\notag\\
		&\le K_g^2 \sum_{i=1}^{N} \sum_{j=1}^{N} \left\|\mu_i-\mu_j\right\|^2\notag\\
		&\le K_g^2 \sum_{i=1}^{N} \sum_{j=1}^{N} \left\|\mu_i-\mu_j\right\|_1^2.
	\end{align}
	
	Since $\mathcal{G}$ is connected,  there always exists a path connecting nodes $i_0$ and $j_0$ for any $i_0,j_0\in\mathcal{V}$, i.e.,
	\begin{align}
		h(\mu)^2=\left(\sum_{i=1}^N\sum_{j=1}^Na_{ij}\left\|\mu_i-\mu_j\right\|_1\right)^2\ge \left\|\mu_{i_0}-\mu_{j_0}\right\|_1^2.
	\end{align}
	
	Then for $K_\mu \ge NK_g$, one has that
	$\phi(\mu)^2\le K_\mu^2 h(\mu)^2$,
	i.e.,
	\begin{align}
		\phi(\mu)-K_\mu h(\mu)\le 0.\label{th1c}
	\end{align}

	Substituting (\ref{th1a}) and (\ref{th1c}) into (\ref{th1b}) completes the proof.
	
\end{proof}

\begin{theorem}\label{th2}	
	Suppose that Assumptions \ref{asp1}-\ref{asp4} hold. Then for any $T\ge0$ and $\varepsilon>0$ in control protocol (\ref{eqpi1}), by choosing $K_\mu \ge NK_g$, it holds for the fit that
	\begin{align}
		\mathcal{F}^T\le \frac{\sqrt{N}\|y(0)-y^*\|}{\varepsilon}  +  2N\sqrt{\frac{K_f}{\varepsilon}}\sqrt{T}.
	\end{align}
	
\end{theorem}

\begin{proof}
	By Lemma \ref{lemma4} with $\tilde{y}=y^*$ and $\tilde{\mu}= \boldsymbol{1}_N \otimes \gamma$, where $\gamma=col(\gamma_1,\dots,\gamma_q)\in\mathbb{R}^q$ is a parameter to be determined later, one has that
	\begin{align}
		&\int_{0}^{T} \Big(\mathcal{H}(t,y,\boldsymbol{1}_N\otimes \gamma) - \mathcal{H}(t,y^*,\mu)\Big) \,dt\notag\\
		=
		&\int_{0}^{T} \bigg(f(t,y){+}\gamma^\top \sum_{i=1}^{N}g_i(t,y_i){-}f(t,y^*){-}\mu^\top g(t,y^*) \notag\\
		&+K_\mu h(\mu) \bigg) dt \notag\\
		\le
		& \frac{V_{(\tilde{y}, \tilde{\mu})}(y(0), \boldsymbol{0})}{\varepsilon}.\label{th2a}
	\end{align}
	Invoking Assumption \ref{asp2} yields
	\begin{align}
		\int_{0}^{T} \big(f(t,y^*)-f(t,y) \big)\,dt \le 2NK_fT.\label{th2b}
	\end{align}
	
	Substituting (\ref{th1c}) and (\ref{th2b}) into (\ref{th2a}), one has that
	\begin{align}
		\int_{0}^{T} \gamma^\top \sum_{i=1}^{N}g_i(t,y_i)\,dt \le& \frac{V_{(\tilde{y}, \tilde{\mu})}(y(0), \boldsymbol{0})}{\varepsilon} + 2NK_fT.\notag
	\end{align}
	
	By choosing
	\begin{align}
		\gamma_j=\left\{
		\begin{aligned}
			& 0, &F_j^T\le 0 \\
			& \frac{\varepsilon}{N}F_j^T, &F_j^T>0
		\end{aligned}
		\right.,j=1,\dots,q,\notag
	\end{align}
	it can be concluded that
	\begin{align}
		\frac{\varepsilon}{N}\|\mathcal{F}^T\|^2\le &\frac{\|y(0)-y^*\|^2 + \frac{\varepsilon^2}{N}\|\mathcal{F}^T\|^2}{2\varepsilon} + 2NK_fT.\notag
	\end{align}
	
	It can be further obtained by transposition that
	\begin{align}
		\mathcal{F}^T \le& \frac{\sqrt{N}\|y(0)-y^*\|_2}{\varepsilon}  +  2N\sqrt{\frac{K_f}{\varepsilon}}\sqrt{T}.
	\end{align}	
\end{proof}

\begin{remark}\label{remark2}
	
	Theorems \ref{th1} and \ref{th2} mean that $\mathcal{R}^T=\mathcal{O}(1)$ and $\mathcal{F}^T = \mathcal{O}(\sqrt{T})$ under continuous communication.
	In comparison, explicit bounds on both the regret and fit with a sublinear growth are obtained in \cite{yi2020,li2022} for single-integrator multi-agent systems, i.e., $\mathcal{R}^T=\mathcal{O}(T^{\max\{\kappa,1-\kappa\}})$, $\mathcal{F}^T = \mathcal{O}(T^{\max\{\kappa,1-\kappa\}})$ in \cite{yi2020} and $\mathcal{R}^T= \mathcal{O}(T^{\max\{\kappa,1-\kappa\}})$, $\mathcal{F}^T = \mathcal{O}(T^{\max\{\frac{1}{2}+\frac{\kappa}{2},1-\frac{\kappa}{2}\}})$ in \cite{li2022} for $\kappa\in(0,1)$. Theorem \ref{th1} achieves stricter regret bound than \cite{yi2020,li2022} under more complex system dynamics.
\end{remark}

\subsection{Event-triggered Communication}
The above continuous-time control law, which requires each agent to know the real-time Lagrange multipliers of neighbors, may cause excessive communication overhead. 
In this section, an event-triggered protocol is proposed to avoid continuous communication.

For agent $i$, suppose that $t_i^l$ is its $l$th communication instant and $\{t_i^1, \dots, t_i^l, \dots\}$ is its communication instant sequence. Define $\hat{\mu}_i(t)\!:=\! \mu_i(t_i^l), ~\forall t\!\in\! [t_i^l,t_i^{l+1})$ as the available information of its neighbors and $e_i\!:=\!\hat \mu_i(t)-\mu_i(t)$ as the measurement error. It can be known  that $e_i=0$ at any instant  $t_i^l$.

An event-triggered control law is proposed as
\begin{subequations}\label{eqpi4}
	\begin{align}
		u_i&=-K_{\alpha_i} x_i+K_{\beta_i} \left(\Pi_{\mathcal{Y}_i}[y_i, -\varepsilon\mathcal{H}_i^{y_i}(t,y_i,\mu_i)]\right),\label{eqpi4a}\\
		\dot{\mu_i}&=\Pi_{\mathbb{R}_+^q}[\mu_i,\varepsilon g_i(t,y_i)-2\varepsilon K_\mu \sum_{j=1}^N a_{ij}\sign{\hat\mu_i-\hat\mu_j}], \label{eqpi4b}
	\end{align}
\end{subequations}
where $\varepsilon>0$ is the step size, $K_{\alpha_i},K_{\beta_i}$ are feedback matrices that are solutions of (\ref{eqle3}), $a_{ij}$ is the weight corresponding to the edge $(j,i)$, and the initial value $\mu_i(0)=\boldsymbol{0}$. Note that $0$ is chosen for the sign function in (\ref{eqpi4b}) when its argument is zero.

Substituting controller (\ref{eqpi4}) into (\ref{eqsys}), the system dynamics of the $i$th agent becomes
\begin{subequations}\label{eqpi5}
	\begin{align}
		\dot x_i&=(A_i{-}B_i K_{\alpha_i})x_i+B_iK_{\beta_i}\left(\Pi_{\mathcal{Y}_i}[y_i, -\varepsilon\mathcal{H}_i^{y_i}(t,y_i,\mu_i)]\right),\label{eqpi5a}\\
		\dot{\mu_i}&=\Pi_{\mathbb{R}_+^q}[\mu_i, \varepsilon g_i(t,y_i) {-}2\varepsilon K_\mu \!\sum_{j=1}^N a_{ij}\sign{\hat\mu_i-\hat\mu_j}], \label{eqpi5b}\\
		y_i&=C_ix_i.\label{eqpi5c}
	\end{align}
\end{subequations}

The communication instant is chosen as
\begin{align}
	t_i^{l+1}\!:=\!\inf_{t>t_i^l}\!\Big\{t\Big|\|e_i(t)\|{\ge}\frac{1}{6N\!\sqrt{q}}\!\sum_{j=1}^N \!a_{ij} \!\|\hat\mu_i{-}\hat\mu_j\|_{\!1} {+} \frac{\sigma e^{-\iota t}}{3N^{\!2}\!K_{\!\mu}\!\sqrt{q}}\Big\},\label{tau2}
\end{align}
where $\sigma$ and $\iota$ are prespecified positive real numbers. 

The following lemma is a modification of Lemma \ref{lemma4} under event-triggered communication.
\begin{lemma}\label{lemma6}
	If Assumptions \ref{asp1}-\ref{asp4} hold and $\tilde{\mu}= \boldsymbol1_N \otimes \gamma, \forall \gamma\in \mathbb{R}_+^q$, then for any $T\ge0$ the trajectories of linear multi-agent system (\ref{eqsys}) with control protocol (\ref{eqpi5}) satisfy
	\begin{align}
		\int_{0}^{T} \Big(\mathcal{H}(t,y,\tilde{\mu}){-}\mathcal{H}(t,\tilde{y},\mu)\!\Big)dt{\le}\frac{V_{(\tilde{y}, \tilde{\mu})}(y(0), \boldsymbol{0})}{\varepsilon} {+} \frac{\sigma}{\iota}.
	\end{align}
\end{lemma}
\begin{proof}
	Similar to (\ref{v11}), one can obtain that
	\begin{align}
		\dot{V}_{(\tilde{y}, \tilde{\mu})}
		= 
		&(\mu-\tilde{\mu})^\top \dot{\mu} + (y-\tilde{y})^\top \dot{y}\notag\\
		=
		&\!\sum_{i=1}^{N} \!(\mu_i{-}\tilde{\mu}_i\!)\!^\top \!\Pi_{\mathbb{R}_+^q}\!\Big[\mu_i,\varepsilon g_i(t,y_i) {-}2\varepsilon K\!_\mu\! \sum_{j=1}^N \!a_{i\!j}\sign{\hat\mu_i{-}\hat\mu_j}\!\Big] \notag\\
		& \!+\! \sum_{i=1}^{N} (y_i\!-\!\tilde{y}_i\!)\!^\top \Pi_{\mathcal{Y}_i}\big[y_i, \!-\varepsilon\mathcal{H}_i^{y_i}(t,\!y_i,\!\mu_i)\big] \notag\\
		\le
		& \underbrace{\varepsilon \sum_{i=1}^{N} (\mu_i{-}\tilde{\mu}_i\!)\!^\top\! \mathcal{H}_i^{\mu_i}(t,\!y_i,\!\mu_i)
		{-}\varepsilon \!\sum_{i=1}^{N} (y_i\!-\!\tilde{y}_i\!)\!^\top\! \mathcal{H}_i^{y_i}(t,\!y_i,\!\mu_i)}_{S_{1}}\notag\\
		& \underbrace{ -2\varepsilon K_\mu \sum_{i=1}^{N} \!(\mu_i{-}\tilde{\mu}_i\!)\!^\top \sum_{j=1}^N a_{ij}\sign{\hat\mu_i-\hat\mu_j} }_{S_{2}} \notag\\
		&+ \underbrace{ \varepsilon K_\mu \sum_{i=1}^{N} \!(\mu_i{-}\tilde{\mu}_i\!)\!^\top  \sum_{j=1}^N a_{ij}\sign{\mu_i-\mu_j} }_{S_{3}}.	\label{vb1}
	\end{align}
	
	Since the Lagrangian (\ref{la}) is convex with respect to $y_i$ and concave with respect to $\mu_i$, one has that
	\begin{align}
		S_{1} \le & \varepsilon \sum_{i=1}^{N} \Big(\mathcal{H}_i (t,\tilde{y}_i,\mu_i) - \mathcal{H}_i (t,y_i,\mu_i)\Big)\notag\\  &+\varepsilon \sum_{i=1}^{N} \Big(\mathcal{H}_i(t,y_i,\mu_i) - \mathcal{H}_i(t,y_i,\tilde{\mu}_i)\Big) \notag\\
		= &\varepsilon \Big(\mathcal{H}(t,\tilde{y},\mu) - \mathcal{H}(t,y,\tilde{\mu})\Big).\label{vb2}
	\end{align}
	
	Since $\tilde{\mu}= \boldsymbol1_N \otimes \gamma, \forall \gamma\in \mathbb{R}_+^q$ and graph $\mathcal{G}$ is undirected, it follows that
	\begin{align}
		S_{2}=
		&-2\varepsilon K_\mu \sum_{i=1}^{N} \sum_{j=1}^N a_{ij}\mu_i^\top\sign{\hat\mu_i-\hat\mu_j}\notag\\
		=
		&-\varepsilon K_\mu \sum_{i=1}^{N}  \sum_{j=1}^N a_{ij}\mu_i^\top\sign{\hat\mu_i-\hat\mu_j}\notag\\
		&- \varepsilon K_\mu \sum_{j=1}^{N}  \sum_{i=1}^N a_{ji}\mu_j^\top\sign{\hat\mu_j-\hat\mu_i} \notag\\
		=
		&-\varepsilon K_\mu \sum_{i=1}^{N}  \sum_{j=1}^N a_{ij}(\mu_i-\mu_j)^\top \sign{\hat\mu_i-\hat\mu_j}\notag\\
		= 
		&-\varepsilon K_\mu \sum_{i=1}^{N}  \sum_{j=1}^N a_{ij} (\hat\mu_i-\hat\mu_j-e_i+e_j)^\top \sign{\hat\mu_i-\hat\mu_j}\notag\\
		\le
		& \!-\!\varepsilon K_\mu \!\sum_{i=1}^{N}\!\sum_{j=1}^N a_{ij} \|\hat\mu_i{-}\hat\mu_j\|_1 \!+\! \varepsilon K_\mu\!\sqrt{q} \sum_{i=1}^{N}  \!\sum_{j=1}^N \|e_i{-}e_j\|,	\label{vb3}
	\end{align}
	where the last inequality holds since the relationship  $\|v\|_1\le\sqrt{q}\|v\|, \forall v\in\mathbb{R}^q$.
	
	Similarly, one has that
	\begin{align}
		S_{3}
		= 
		&\varepsilon K_\mu \sum_{i=1}^{N} \!\mu_i^\top  \sum_{j=1}^N a_{ij}\sign{\mu_i-\mu_j}\notag\\
		=
		&\frac{\varepsilon K_\mu}{2} \sum_{i=1}^{N} \sum_{j=1}^N a_{ij}\|\mu_i-\mu_j\|_1\notag\\
		=
		&\frac{\varepsilon K_\mu}{2} \sum_{i=1}^{N} \sum_{j=1}^N a_{ij}\|\hat\mu_i-\hat\mu_j-e_i+e_j\|_1\notag\\	
		\le	
		&\frac{\varepsilon K_\mu}{2} \sum_{i=1}^{N} \sum_{j=1}^N a_{ij}\|\hat\mu_i{-}\hat\mu_j\|_1 + \frac{\varepsilon K_\mu\sqrt{q}}{2} \sum_{i=1}^{N} \sum_{j=1}^N \|e_i{-}e_j\|.\label{vb4}
	\end{align}

	For the second item in (\ref{vb3}) and (\ref{vb4}), one can calculate that

	\begin{align}
		&\sum_{i=1}^{N}  \sum_{j=1}^N  \|e_i-e_j\|\notag\\
		\le& \sum_{i=1}^{N}  \sum_{j=1}^N  \|e_i\| + \sum_{i=1}^{N}  \sum_{j=1}^N  \|e_j\|\notag\\
		\le &2N\sum_{i=1}^{N} \|e_i\|\notag\\
		\le &\frac{1}{3\sqrt{q}}\sum_{i=1}^{N} \sum_{j=1}^N a_{ij} \|\hat\mu_i-\hat\mu_j\|_1 + \frac{2\sigma e^{-\iota t}}{3K_\mu\sqrt{q}} ,	
	\end{align}
	where the last inequality holds due to the trigger condition (\ref{tau2}).
		
	Summarizing the above-discussed analysis,	$\dot{V}_{(\tilde{y}, \tilde{\mu})}$ in (\ref{vb1}) is calculated as
	\begin{align}
		\dot{V}_{(\tilde{y}, \tilde{\mu})}
		\le
		&S_{1}+S_{2}+S_{3}\notag\\
		\le
		& \varepsilon \Big(\mathcal{H}(t,\tilde{y},\mu) - \mathcal{H}(t,y,\tilde{\mu})\Big) + \varepsilon\sigma e^{-\iota t}.\notag
	\end{align}

	By integrating it from $0$ to $T$ on both sides and omitting negative terms, it can be obtained that
	\begin{align}
		\int_{0}^{T} \Big( \mathcal{H}(t,y,\tilde{\mu}){-}\mathcal{H}(t,\tilde{y},\mu)\Big) dt {\le} \frac{V_{(\tilde{y}, \tilde{\mu})}(y(0), \boldsymbol{0})}{\varepsilon} {+} \frac{\sigma}{\iota}.
	\end{align}
\end{proof}

We now state the main results about the regret and fit bounds of event-triggered communication controller (\ref{eqpi5}).

\begin{theorem}\label{th3}
	Suppose that Assumptions \ref{asp1}-\ref{asp4} hold. Then for any $T\ge0$ and $\varepsilon>0$ in control protocol (\ref{eqpi4}), by choosing $K_\mu \ge NK_g$, the following regret bound holds:
	\begin{align}
		\mathcal{R}^T\le \frac{\|y(0)\!-\!y^*\|^2}{2\varepsilon} +\frac{\sigma}{\iota}.
	\end{align}
\end{theorem}

\begin{theorem}\label{th4}	
	Suppose that Assumptions \ref{asp1}-\ref{asp4} hold. Then for any $T\ge0$ and $\varepsilon>0$ in control protocol (\ref{eqpi4}), by choosing $K_\mu \ge NK_g$, the following fit bound holds:
	\begin{align}
		\mathcal{F}^T\le \frac{\sqrt{N}\|y(0)-y^*\|}{\varepsilon}  + \sqrt{\frac{2N\sigma}{\varepsilon\iota}} + 2N\sqrt{\frac{K_f}{\varepsilon}}\sqrt{T}.
	\end{align}
	
\end{theorem}

The proofs of Theorems \ref{th3} and \ref{th4} are similar to that of Theorems \ref{th1} and \ref{th2}, except that Lemma \ref{lemma6} is used instead of Lemma \ref{lemma4}. They are thus omitted here.

\begin{remark}\label{remark3}
	Theorems \ref{th3} and \ref{th4} mean that $\mathcal{R}^T=\mathcal{O}(1)$ and $\mathcal{F}^T = \mathcal{O}(\sqrt{T})$ still hold even under event-triggered communication. The bounds of regret and fit are determined by the communication frequency. Generally speaking, decreasing $\sigma$ and increasing $\iota$ will achieve smaller bounds on regret and fit, but meanwhile increase the communication frequency, which results in a tradeoff between them.	
\end{remark}

\begin{theorem}\label{th5}
	Under the event triggering condition (\ref{tau2}), system (\ref{eqpi5}) does not exhibit the Zeno behavior.. 
\end{theorem}
\begin{proof}
	In the trigger interval $[t_i^l,t_i^{l+1})$ for agent $i$, combining the definition of $e_i$ with (\ref{eqpi4b}), one can write the upper right-hand Dini derivative as
	\begin{align}
		D^+ e_i(t){=} - \!\Pi_{\mathbb{R}_+^q}[\mu_i, \varepsilon g_i(t,y_i) {-}2\varepsilon K_\mu \sum_{j=1}^N a_{ij}\sign{\hat\mu_i{-}\hat\mu_j}].\label{zeno1}
	\end{align}
	It is obvious that $e_i(t_i^l)=0$. Then, for $t\in (t_i^l,t_i^{l+1})$, the solution of (\ref{zeno1}) is
	\begin{align}
		e_i(t){=} \!\int_{t_i^l}^{t} \!- \!\Pi_{\mathbb{R}_+^q}[\mu_i, \varepsilon g_i(t,y_i) {-}2\varepsilon K_\mu \sum_{j=1}^N a_{ij}\sign{\hat\mu_i{-}\hat\mu_j}] \,d\tau. \label{zeno2}
	\end{align}
	From Assumption \ref{asp2} and the inequality $\|\Pi_{\mathcal{S}}[x,v]\|\le\|v\|$ (cf. Remark 2.1 in \cite{zhang1995}), it can be obtained that the norm of the integral term in (\ref{zeno2}) is bounded, and let the upper bound of its norm be $\delta$.  It then follows from (\ref{zeno2}) that
	\begin{align}
		\|e_i(t)\| \le \delta(t-t_i^l).
	\end{align} 
	Hence, condition (\ref{tau2}) will definitely not be triggered before the following  condition holds:
	\begin{align}
		\delta(t-t_i^l)=\frac{1}{6N\sqrt{q}}\sum_{j=1}^N a_{ij} \|\hat\mu_i-\hat\mu_j\|_1 {+} \frac{\sigma e^{-\iota t}}{3N^2K_\mu\sqrt{q}} .\label{zeno3}
	\end{align}
	
	It is easy to obtain that the right-hand side of (\ref{zeno3}) is positive for any finite time $t$,  which further implies that $t-t_i^l>0$. Hence, the value $t_i^{l+1}-t_i^l$ is strictly positive for finite $t$,
	which implies that no Zeno behavior is exhibited.
\end{proof}

\section{SIMULATION}\label{sim}

Consider a heterogeneous multi-agent system composed of $5$ agents described by (\ref{eqsys}), where 
$x_i\in\begin{cases}
	\mathbb{R}^2 &i=1,2,3\\\mathbb{R}^3 &i=4,5
\end{cases}$,
$y_i=\begin{cases}
	(y_{i,a},y_{i,b})\in\mathbb{R}^2 &i=1,2,3\\(y_{i,a},y_{i,b},y_{i,c})\in\mathbb{R}^3 &i=4,5
\end{cases}$,
 $A_{1,2}$ $=$ $[1,0;0,2]$, $A_{3}$ $=$ $[0,2;-1,1]$, $A_{4,5}$ $=$ $[2,1,0;0,1,1;1,0,2]$, $B_{1,2}$ $=$ $[0,1;1,3]$, $B_{3}$ $=$ $[2,1;1,0]$, $B_{4,5}$ $=$ $[1,0,0;0,1,0;0,0,1]$, $C_{1,2}$ $=$ $[2,0;0,1]$, $C_{3}$ $=$ $[2,1;-1,0]$, $C_{4,5}$ $=$ $[3,0,0;0,1,0;0,1,2]$.

\begin{figure}[htbp]
	\centering
	\includegraphics[width=0.28\textwidth]{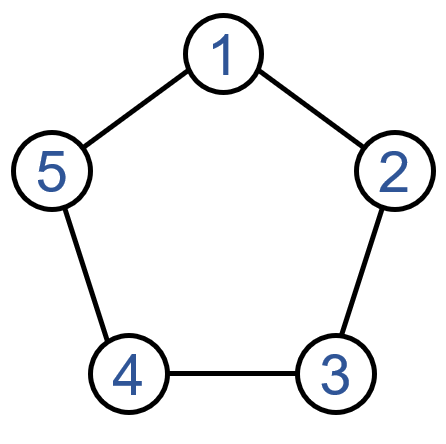} 
	\caption{Communication network among five agents.}
	\label{pic1}
\end{figure}

\begin{figure}[htbp]
	\centering
	\includegraphics[width=0.45\textwidth]{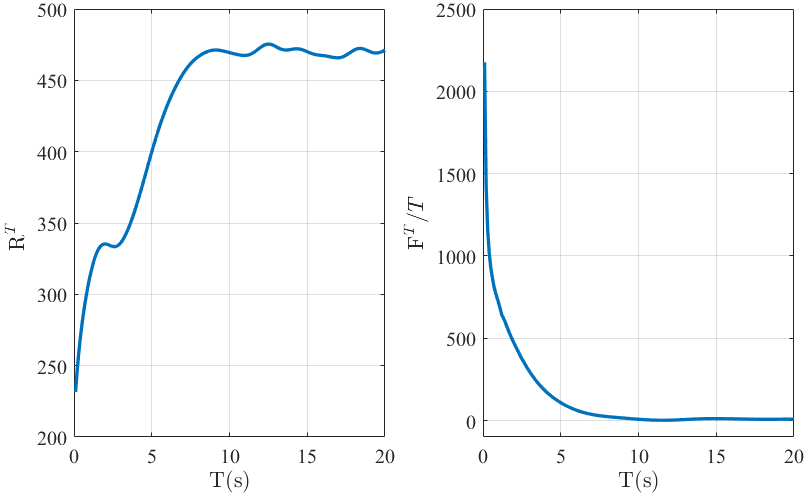} 
	\caption{ Evolution of $\mathcal{R}^T$ and $\mathcal{F}^T/T$ with continuous communication.}
	\label{pic3}
\end{figure}

The local objective functions are time-varying quadratic functions as follows:
\begin{align*}
	f_1(t,y_1)=&(y_{1,a}{-}\cos{t}{-}1)^2{+}(y_{1,b}{-}\cos{1.5t}{-}1.5)^2;\\
	f_2(t,y_2)=&2(y_{2,a}{-}\cos{t}{-}1)^2{+}3(y_{2,b}{-}\cos{1.7t}{-}1)^2;\\
	f_3(t,y_3)=&2(y_{3,a}{-}\cos{t}{-}1)^2{+}(y_{3,b}{-}\cos{2t}{-}1)^2;\\
	f_4(t,y_4)=&0.5(y_{4,a}{-}\cos{t}{-}2)^2{+}(y_{4,b}{-}\cos{1.2t}{-}1)^2\\
	&{+}(y_{4,c}{-}\cos{1.5t}{-}2)^2;\\
	f_5(t,y_5)=&2(y_{5,a}{-}\cos{t}{-}1)^2{+}3(y_{5,b}{-}\cos{1.5t}{-}1)^2\\
	&{+}(y_{5,c}{-}\cos{2t}{-}1.5)^2.
\end{align*}

The feasible set of output variables $\mathcal{Y}\in[-1,5]^{12}$. The constraints are defined by a time-varying function
\begin{align*}
	g_1(t,y_1)=&(0.5\sin{10t}{+}1.5)y_{1,a}{+}(0.5\sin{15t}{+}1.5)y_{1,b}{-}1;\\
	g_1(t,y_2)=&(0.3\sin{10t}{+}1.7)y_{2,a}{+}(0.1\sin{25t}{+}1.9)y_{2,b}{-}3;\\
	g_1(t,y_3)=&(0.6\sin{20t}{+}1.4)y_{3,a}{+}(0.5\sin{15t}{+}1.5)y_{3,b}{-}4;\\
	g_1(t,y_4)=&(0.5\sin{20t}{+}1.5)y_{4,a}{+}(0.4\sin{25t}{+}1.6)y_{4,b}\\
	&{+}(0.4\sin{25t}{+}1.6)y_{4,c}{-}2;\\
	g_1(t,y_5)=&(0.5\sin{10t}{+}1.5)y_{5,a}{+}(0.6\sin{15t}{+}1.4)y_{5,b}\\
	&{+}(0.6\sin{15t}{+}1.4)y_{5,c}{-}5;
\end{align*}
 The above constraint selection ensures that $\boldsymbol{0}$ must be a strictly feasible solution for all $t\in[0,T]$.

The clairvoyant optimal output can be computed by solving the problem
\begin{align}
	\begin{split}
		y^* \!:= \mathop{\arg\min}_{y \in \mathcal{Y}} \int_{0}^{T} \sum_{i=1}^{N}f_i(t,y_i)\,dt ,\\
		s.t. \sum_{i=1}^{N}g_{i}(t, y_i)\le \boldsymbol0.\label{question4}
	\end{split}
\end{align}

The communication network among these agents is depicted in Fig. \ref{pic1}.
It can be verified that Assumptions \ref{asp1}$-$\ref{asp4} hold.

For the numerical example, the selection of feedback matrices is based on (\ref{eqle3}), where $K_{\alpha_{1,2}}$ $=$ $[-3,2;1,0]$, $K_{\alpha_{3}}$ $=$ $[-1,1;2,0]$, $K_{\alpha_{4,5}}$ $=$ $[2,1,0;0,1,1;1,0,2]$, $K_{\beta_{1,2}}$ $=$ $[-1.5,1;0.5,0]$, $K_{\beta_{3}}$ $=$ $[1,2;-2,-5]$, $K_{\beta_{4,5}}$ $=$ $[0.333,0,0;0,1,0;0,-0.5,0.5]$. The initial values $x_i(0)$ are randomly selected in $[-5, 5]$ and $\mu(0)=\boldsymbol{0}$.

\begin{figure}[tbp]
	\centering
	\includegraphics[width=0.45\textwidth]{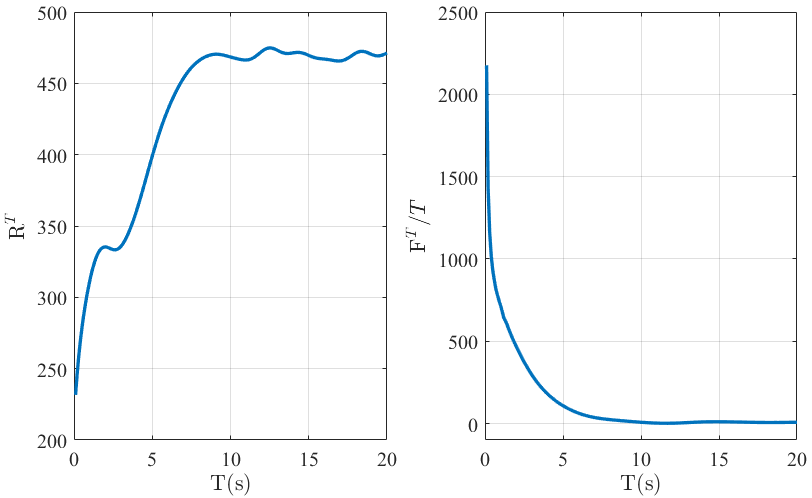} 
	\caption{Evolution of $\mathcal{R}^T$ and $\mathcal{F}^T/T$ with event-triggered communication.}
	\label{pic4}
\end{figure}

\begin{figure}[tbp]
	\centering
	\includegraphics[width=0.4\textwidth]{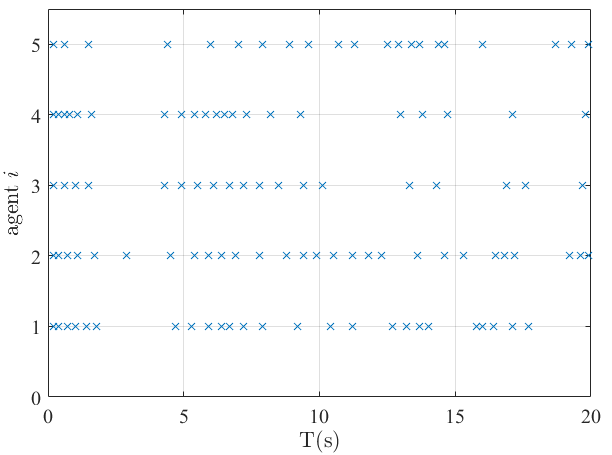} 
	\caption{ Triggering instants of five agents.}
	\label{pic5}
\end{figure}

Fig. \ref{pic3} illustrates that the continuous-time control law achieves constant regret bound and sublinear fit bound. The results are in accordance with those established in Theorems \ref{th1}-\ref{th2}. Likewise, the similar results can be observed in Fig. \ref{pic4} under event-triggered communication. Fig. \ref{pic5} shows the communication moments of five agents with event-triggered control laws, from which one can observe that the communication among five agents is discrete and exhibits no Zeno behavior.

\section{CONCLUSION}\label{res}
In this paper, we studied distributed online convex optimization for heterogeneous linear multi-agent systems with time-varying cost functions and time-varying coupling inequality constraints. A distributed controller was proposed based on the saddle-point method, showing the constant regret bound and sublinear fit bound.  In order to avoid continuous communication and reduce the communication cost, an event-triggered communication scheme with no Zeno behavior was developed, which also achieves constant regret bound and sublinear fit bound.

\bibliographystyle{ieeetr}
\bibliography{bibfile}  

\end{document}